\documentclass{amsart}
\makeatletter
\def\LaTeX{\leavevmode L\raise.42ex
   \hbox{\kern-.3em\size{\sf@size}{0pt}\selectfont A}\kern-.15em\TeX}
\makeatother

\newcommand{\BibTeX}{{\rm B\kern-.05em{\sc
i\kern-.025emb}\kern-.08em\TeX}}
\newtheorem{col}{Corollary}[section]
\newtheorem{thm}{Theorem}[section]
\newtheorem{lem}[thm]{Lemma}

\theoremstyle{definition}
\newtheorem{defn}{Definition}

\numberwithin{equation}{section}

\begin{document}

\title[  Bernstein-Nikolskii and Plancherel-Polya inequalities  on 
 symmetric spaces  ]{Bernstein-Nikolskii and Plancherel-Polya inequalities in $L_{p}$-norms on non-compact
 symmetric spaces}

\author{Isaac Pesenson}
\address{Department of Mathematics, Temple University,
Philadelphia, PA 19122} \email{pesenson@math.temple.edu}


 \begin{abstract}

 By using Bernstein-type inequality we define analogs of spaces of entire
 functions of exponential type in $L_{p}(X), 1\leq p\leq \infty$,
 where $X$ is a symmetric  space of non-compact.
We give estimates of $L_{p}$-norms, $1\leq p\leq \infty$, of such
functions (the Nikolskii-type inequalities) and also prove the
$L_{p}$- Plancherel-Polya inequalities which imply  that our
functions of exponential type are uniquely determined by their
inner products with certain countable sets of measures with
compact supports and can be reconstructed from such sets of
"measurements" in a stable way.

\end{abstract}

\maketitle

 \section{Introduction and Main Results}

Consider the subspace $E^{\omega}_{p}(\mathbb{R}^{d}), 1\leq p\leq
\infty, \omega\geq 0,$ of $L_{p}(\mathbb{R}^{d})$ which consists
of all functions which have extension to $\mathbb{C}^{d}$ as
entire functions of exponential type $\leq \omega$. The latest
means that for any $\varepsilon>0$ there exists a
$C_{\varepsilon}>0$ such that
$$
|f(z_{1},z_{2},...,z_{d})|\leq
C_{\varepsilon}e^{\sum_{j=1}^{d}(\omega+\varepsilon)|z_{j}|},
$$
where $f\in E^{\omega}_{p}(\mathbb{R}^{d}),
(z_{1},z_{2},...,z_{d})\in \mathbb{C}^{d}$.

A function $f$ belongs to the space
$E^{\omega}_{p}(\mathbb{R}^{d}), 1\leq p\leq \infty,$ if and only
if it satisfies the Bernstein inequality
$$
\left\|\frac{\partial^{k} f}{\partial x_{j_{1}}...
\partial
x_{j_{k}}}\right\|_{L_{p}(\mathbb{R}^{d})}\leq
\omega^{k}\left\|f\right\|_{L_{p}(\mathbb{R}^{d})},
$$
for any sequence $1\leq j_{1},...,j_{k}\leq d.$ The Paley-Wiener
theorem says that the distributional Fourier transform
$$
\hat{f}(\xi)=\frac{1}{(2\pi)^{d/2}}\int_{\mathbb{R}^{d}}e^{-i\xi
x}f(x)dx
$$
of a function from a space $E^{\omega}_{p}(\mathbb{R}^{d}), 1\leq
p\leq \infty,$ has support in the cube
$$
Q_{\omega}=\{|\xi_{j}|\leq \omega, j=1,2,...,d\}.
$$

The following inequality plays an important role in approximation
theory and in the
 theory of function spaces \cite{N}, \cite{T1}, \cite{T2}, and is
 known as the Nikolskii inequality.

\begin{equation}
\|f\|_{L_{q}(\mathbb{R}^{d})}\leq
2^{d}\omega^{\frac{d}{p}-\frac{d}{q}}\|f\|_{L_{p}(\mathbb{R}^{d})},1\leq
p\leq q\leq\infty,
\end{equation}
where $f\in E^{\omega}_{p}(\mathbb{R}^{d})$. Note that this
inequality is exact  in the sense that for the function
$$
f(x_{1},x_{2},...,x_{d})=\prod_{i=1}^{d}x_{i}^{-2}\sin^{2}\frac{\omega
x_{i}}{2}
$$
one has the equality
$$
\|f\|_{L_{q}(\mathbb{R}^{d})}=C(p,q)\omega^{\frac{d}{p}-\frac{d}{q}}
\|f\|_{L_{p}(\mathbb{R}^{d})},
$$
for any $1\leq p\leq q\leq \infty$. The detailed proofs of all
these results can be found in \cite{Akh}, \cite{B}, \cite{N},
\cite{PW}.

It is also known that there exist two positive constants $C_{1},
C_{2}$ such that for any sufficiently dense  discrete set of
points $\{x_{j}\}\in \mathbb{R}^{d}$ and any $f\in
E^{\omega}_{p}(\mathbb{R}^{d})$ the following
Plancherel-Polya-type inequalities hold true  \cite{SS},
\cite{T2},
$$
 C_{1}\left(\sum_{j\in
\mathbb{N}}|f(x_{j})|^{p}\right)^{1/p}\leq
\|f\|_{L_{p}(\mathbb{R}^{d})}\leq C_{2}\left(\sum_{j\in
\mathbb{N}}|f(x_{j})|^{p}\right)^{1/p}.
 $$

 The Plancherel-Polya inequalities imply that every function from
 $E^{\omega}_{p}(\mathbb{R}^{d}), 1\leq p\leq \infty, \omega>0, $
 is uniquely determined by its values on sufficiently dense (depending on $\omega$)
  discrete
 sets of points $\{x_{j}\}\in\mathbb{R}^{d}$ and can be
 reconstructed from these values in a stable way.

 The goal of this paper is to develop similar theory in $L_{p}(X), 1\leq p\leq\infty,$
 where $X$ is a  non-compact symmetric manifold.
 In the case $p=2$ it was partially done
 in our previous papers \cite{Pes3}-\cite{Pes6}.

In the sections 2-4 we consider a symmetric space of non-compact
type $X=G/H$ where $G$ is a semi-simple Lie group with finite
center and $H$ its maximal compact subgroup \cite{H1}, \cite{H2}.
The elements of the corresponding Lie algebra $\textbf{g}$ of $G$
will be identified with left-invariant vector fields on $G$. The
action of the group $G$ on functions defined on the space $X$ is
given by the formula
\begin{equation}
T_{g}f(x)=f(g\cdot x), g\in G, x\in X=G/H.
\end{equation}

The corresponding representation of the group $G$ in any space
$L_{p}(X), 1\leq p\leq \infty,$ is known as quasi-regular
representation.

One can consider the so-called differential associated with action
(1.2). The differential is a map from the Lie algebra $\textbf{g}$
 into algebra of differential operators on the
space $X$. We will use the same notation for elements of
$\textbf{g}$ and their images under the differential.

 Using a set of vector
fields $\mathbb{V}$ we define the sets of functions
$E_{p}^{\omega}(\mathbb{V}), 1\leq p\leq \infty, \omega>0,$ as the
sets of all $f\in L_{p}(X), 1\leq p\leq \infty, $ for which the
following  Bernstein inequality holds true
\begin{equation}
\|V_{i_{1}}V_{i_{2}}...V_{i_{k}}f\|_{L_{p}(X)}\leq
\omega^{k}\|f\|_{L_{p}(X)}, k\in \mathbb{N}.
\end{equation}

It is not clear from this definition  if the set
 $E_{p}^{\omega}(\mathbb{V})$ is  linear. It becomes obvious after
 we prove that  this set coincide with
the set of all functions $f\in  L_{p}(X), 1\leq p\leq \infty,$
such that for any  choice of indices
  $1\leq i_{1},... ,i_{k}\leq d$,  any $1\leq j\leq d$ and any functional
   $h$ on  $L_{p}(X)$ the function
$$
\left<h,e^{tV_{j}}V_{i_{1}}...V_{i_{k}}f\right>:
\mathbb{R}\rightarrow  \mathbb{C},
$$
  of the real variable $t$ is entire function of the exponential
  type $\omega$ which is bounded on the real line.

 Although different bases of vector
fields will produce different scales of spaces in the sense that
for  a particular $\omega$
$$
E^{\omega}_{p}(\mathbb{V})\neq E^{\omega}_{p}(\mathbb{U}),
$$
 their unions
$\bigcup_{\omega>0}E^{\omega}_{p}(\mathbb{V})$ and
$\bigcup_{\omega>0}E^{\omega}_{p}(\mathbb{U})$ will be the same.

The spectral resolution of the Laplace-Beltrami operator $\Delta$
in the space $L_{2}(X)$ is given by the Helgason-Fourier transform
on $X$. The existence of such transform allows to introduce
Paley-Wiener spaces $PW_{\omega}(X)$ as sets of all functions from
$L_{2}(X)$ for which  Helgason-Fourier transform has compact
support bounded by $\omega$ in the non-compact direction (see
below). As a consequence of our general result (Theorem 3.2) about
Paley-Wiener vectors for self-adjoint operators we obtain that a
function $f$ belongs to the Paley-Wiener space $PW_{\omega}(X)$ if
and only if the following Bernstein inequality holds true
$$
\|\Delta^{s/2}f\|_{L_{2}(X)}\leq
\left(\omega^{2}+\|\rho\|^{2}\right)^{s/2}\|f\|_{L_{2}(X)},
$$
where $\rho$ is the half-sum of all positive restricted roots and
its norm is calculated  with respect to the Killiing form.

The fact that the Laplace-Beltrami operator (1.4) commutes with
the fields $V_{1},V_{2},...,V_{d}$  allows to obtain the following
continuous embeddings
$$
E^{\Omega/\sqrt{d}}_{2}(\mathbb{V})\subset PW_{\omega}(X)\subset
E_{2}^{\Omega}(\mathbb{V}), d=dim X,
$$
where $\Omega=\sqrt{\omega^{2}+\|\rho\|^{2}}$. These embeddings
imply that the spaces $E_{2}^{\omega}(\mathbb{V})$ are not trivial
at least if $\omega\geq \|\rho\|$ and their union
$\bigcup_{\omega>0}E_{2}^{\omega}(\mathbb{V})$ is dense in
$L_{2}(X)$.

In the Theorem 4.2  we prove an inequality which is in the case of
$\mathbb{R}^{d}$ is known as the Nikolskii inequality. Namely, we
show  that for any $m>d/p$ and any sufficiently discrete set
$\{g_{i}\}\in G$ there exist constants $C(X), C(X,m)$ such that
\begin{equation}
\|f\|_{L_{q}(X)}\leq C(X)r^{d/p} \sup_{g\in
G}\left(\sum_{i}\left(|f(g_{i}g\cdot
o)|\right)^{p}\right)^{1/p}\leq $$ $$
C(X,m)r^{d/q-d/p}\left(1+(r\omega)^{m}\right)\|f\|_{L_{p}(X)},
\end{equation}
for all functions from $E_{p}^{\omega}(\mathbb{V})$ and $1\leq
p\leq q\leq\infty$. Here $o$ is the "origin" of $X$ and $r$ is
essentially the "distance" between points:
$\sup_{j}\inf_{i}dist(g_{j},g_{i})$. Using these inequalities we
prove the continuous embedding
\begin{equation}
E^{\omega}_{p}(\mathbb{V})\subset E^{\omega}_{q}(\mathbb{V}),
1\leq p\leq q\leq\infty.
\end{equation}

The embedding (1.7) has an important consequence  that the spaces
$E^{\omega}_{q}(\mathbb{V})$ are not trivial at least if $q\geq 2$
and $\omega\geq \|\rho\|$.  This result is complementary to a
 result of the classical paper \cite{EM} of L. Ehrenpreis
and F. Mautner which says that in $L_{1}(X)$ there are not
non-trivial functions whose Helgason-Fourier transform has compact
support. As another consequence of the inequalities (1.6) we
obtain a generalization of  the Nikolskii inequality (1.1) for
functions from $E^{\omega}_{p}(\mathbb{V})$
 $$
 \|f\|_{L_{q}(X)}\leq
C(X)\omega^{\frac{d}{p}-\frac{d}{q}}\|f\|_{L_{p}(X)}, d=dim
X,1\leq p\leq q\leq\infty.
$$

We also prove a generalization of the Plancherel-Polya
inequalities for functions from $E^{\omega}_{p}(\mathbb{V}), 1\leq
p\leq\infty$. We show that there exist constants $C(X), c(X)$ such
that for every $\omega>0$, every "sufficiently dense" discrete set
of measures $\{\Phi_{\nu}\}$ with compact supports the following
inequalities hold true

\begin{equation}
c(X)\left(\sum_{\nu}\left|\Phi_{\nu}(f)\right|^{p}\right)^{1/p}\leq
r^{-d/p}\|f\|_{p} \leq C(X)\left(\sum_{\nu}
|\Phi_{\nu}(f)|^{p}\right)^{1/p}.
\end{equation}
$f\in E^{\omega}_{p}(\mathbb{V}), 1\leq p\leq\infty$ and $r$ is
comparable to the "distance" between supports of distributions
$\{\Phi_{\nu}\}$. The Plancherel-Polya-type inequalities (1.8)
obviously imply that every $f\in E^{\omega}_{p}(\mathbb{V}), 1\leq
p\leq\infty,$ is uniquely determined by the values
$\{\Phi_{\nu}(f)\}$ and can be reconstructed in a stable way.

 Note that an approach to Paley-Wiener functions and the
Bernstein inequality in a Hilbert space in which  a strongly
continuous representation of a Lie group is given  were developed
by author in \cite{Pes1}-\cite{Pes5}.

\section{Bernstein-type inequality in $L_{p}(X), 1\leq
p\leq\infty$.}

A non-compact Riemannian symmetric space $X$ is defined as $G/K$,
where $G$ is a connected non-compact semi-simple group Lie whose
Lie algebra has a  finite center and $K$ its maximal compact
subgroup. Their Lie algebras will be denoted respectively as
$\textbf{g}$ and $\textbf{k}$. The group $G$ acts on $X$ by left
translations and it has the "origin" $o=eK$, where $e$ is the
identity in $G$. Every such $G$ admits so called Iwasawa
decomposition $G=NAK$, where nilpotent Lie group $N$ and abelian
group $A$ have Lie algebras $\textbf{n}$ and $\textbf{a}$
respectively.
 Letter $M$ is usually used to denote the centralizer of $A$ in
$K$ and letter $\mathcal{B}$ is  used for the factor
$\mathcal{B}=K/M$ which is known as a boundary.

The Killing form on $G$ induces an inner product on tangent spaces
of $X$. Using this inner product it is possible to construct
$G$-invariant Riemannian structure on $X$. The Laplace-Beltrami
operator of this Riemannian structure is denoted as $\Delta$.

In particular, if $X$ has rank one ($dim A=1$) then in a polar
geodesic coordinate system $(r,\theta_{1},...,\theta_{d-1})$ on
$X$ at every point $x\in X$ the operator $\Delta$ has the form
\cite{H2}
$$
\Delta=\partial^{2}_{r}+\frac{1}{S(r)}\frac{dS(r)}{dr}\partial
_{r}+\Delta_{S},
$$
where $\Delta_{S} $ is the Laplace-Beltrami operator on the sphere
$S(x,r)$ of the induced Riemannian structure on $S(x,r)$ and
$S(r)$ is the surface area of a sphere of radius $r$ which depends
just on $r$ and is given by the formula
$$
S(r)=\Omega_{d}2^{-b}c^{-a-b}sh^{a}(cr)sh^{b}(2cr),
$$
 where $d=dim
X=a+b+1, c=(2a+8b)^{-1/2}$, $a$ and $b$ depend on $X$ and $
\Omega_{d}=2\pi^{d/2}(\Gamma(d/2))^{-1} $
 is the surface area of the unit sphere in $d$-dimensional
 Euclidean space.

In this section we will use the notation $L_{p}(X), 1\leq
p\leq\infty,$ with understanding that in the case $1\leq p <
\infty$ the space $L_{p}(X)$ represents the usual $L_{p}(X)$ with
respect to the invariant measure $dx$ on $X$ and in the case
$p=\infty$ we have the space  of uniformly continuous bounded
functions on $X$.

 The  goal of the section is to introduce a scale of closed
linear subspaces in $L_{p}(X), 1\leq  p\leq \infty,$ for which an
analog of the Bernstein inequality holds true. In the case $p=2$
our spaces consist of functions whose Helgason-Fourier transform
has compact support (see Section 3). We also show that our spaces
are not trivial at least in the case $p\geq 2, \omega\geq
\|\rho\|$.

Every vector $V$ in the Lie algebra $\textbf{g}$ can be identified
with a left-invariant vector field on $G$ , which will be denoted
by the same letter $V$.

There exists a basis $V_{1},... ,V_{d},..., V_{n} \in \textbf{g},
n=dim G, d=dim X,$ in  $\textbf{g}$ such that $V_{d+1},...,V_{n}$
form a basis of the algebra Lie of the compact group $K$ and such
that
$$
\left<V_{i},V_{j}\right>=\delta_{ij}
$$
if $1\leq i,j\leq d$, and
$$
\left<V_{i},V_{j}\right>=-\delta_{ij}
$$
if $d+1\leq i,j\leq n$, where $\left<,\right>$ is the Killing form
and $\delta_{ij}$ is the Kronecker symbol.

In this basis the differential operator
\begin{equation}
V_{1}^{2}+...+V_{d}^{2}-V_{d+1}^{2}-...-V_{n}^{2},
\end{equation}
on $G$ belongs to the center of the algebra of all left-invariant
differential operators on $G$ and is known as the Casimir
operator.

 We are going to use the same notations
for the vectors $V_{1},...,V_{d}\in \textbf{g}$ and for their
images under the differential of the quasi-regular representation
of $G$ in $L_{p}(X), 1\leq p\leq \infty$.

The image of  every $V_{1},...,V_{d}$ under the differential of
the quasi-regular representation of $G$ in the space $L_{p}(X),
1\leq p\leq \infty,$ is a generator of strictly continuous
isometric one-parameter group in $L_{p}(X), 1\leq p\leq \infty,$
which is given by the formula
$$
e^{tV_{j}}f(x)=f(\exp tV_{j}\cdot x), x\in X, f\in L_{p}(X), 1\leq
p\leq \infty,
$$
where $\exp tV_{j}$ is the flow generated by the vector field
$V_{j}$. In the case $p=2$ these generators are skew-symmetric
operators.

Note that  the Laplace-Beltrami operator $\Delta$ on $X$ commutes
with the operators $V_{1},...,V_{d}$.

\bigskip

In what follows the notation $\|f\|_{p}, 1\leq p\leq \infty,$ will
always mean the norm $\|f\|_{L_{p}(X)}, 1\leq p\leq \infty,$ of a
function $f$.

\begin{defn}
A function $f\in L_{p}(X), 1\leq p\leq \infty, $ belongs to the
set $E_{p}^{\omega}(\mathbb{V})$ if and only if for every $1\leq
i_{1},...i_{k}\leq d$ the following Bernstein inequality holds
true
\begin{equation}
\|V_{i_{1}}...V_{i_{k}}f\|_{p}\leq \omega^{k}\|f\|_{p}.
\end{equation}
\end{defn}

\begin{defn}
The liner space $ \mathbb{E}^{\omega}_{p}(\mathbb{V}), \omega>0, $
is the set of all functions $f\in  L_{p}(X), 1\leq p\leq \infty,$
such that for any
  $1\leq i_{1},... ,i_{k}\leq d$,  any $1\leq j\leq d$ and any functional
   $h\in  L_{p}(X)^{*}$ the function
$$
\left<h,e^{tV_{j}}V_{i_{1}}...V_{i_{k}}f\right>:
\mathbb{R}\rightarrow  \mathbb{C},
$$
  of the real variable $t$ is entire function of the exponential
  type $\omega$.
\end{defn}

We are going to show that  these definitions are equivalent. All
the necessary information about one-parameter groups of operators
can be found in \cite{BB} and \cite{T1}.

\begin{thm}
The sets $E_{p}^{\omega}(\mathbb{V})$ and
$\mathbb{E}_{p}^{\omega}(\mathbb{V}), \omega>0,$ coincide.
\end{thm}
\begin{proof}
Suppose that $f\in E_{p}^{\omega}(\mathbb{V})$, then for any
function $g=V_{i_{1}}...V_{i_{k}}f, 1\leq i_{1},...,i_{k}\leq d,$
and any $1\leq j\leq n$ because we have the estimate (2.2) the
series
\begin{equation}
e^{zV_{j}}g=\sum \frac{(zV_{j})^{r}}{r!}g
\end{equation}
 is convergent in $L_{p}(X)$ and represents
 an abstract entire function. Since
  $\|V_{j}^{r}g\|_{p}\leq
 \omega^{k+r}\|f\|_{p}$ we have the  estimate

$$ \left\|e^{zV_{j}}g\right\|_{p}=
\left\|\sum
^{\infty}_{r=0}\left(z^{r}V_{j}^{r}g\right)/r!\right\|_{p} \leq
\omega^{k}\|f\|_{p}\sum
^{\infty}_{r=0}\frac{|z|^{r}\omega^{r}}{r!}=
\omega^{k}e^{|z|\omega}\|f\|_{p},
$$
which shows that the function (2.3) has exponential type $\omega$.
Since $e^{tV_{j}}$ is a group of isometries, the abstract function
$e^{tV_{j}}g$ is bounded by $\omega^{k}\|f\|_{p}$. It implies that
for any functional $h$ on $ L_{p}(X), 1\leq p\leq \infty,$ the
scalar function
$$
F(z)=\left<h, e^{zV_{j}}g\right>
$$
is entire because it is defined by the series
\begin{equation}
F(z)=\left<h, e^{zV_{j}}g\right>=\sum^{\infty}_{r=0}
\frac{z^{r}\left<h,V_{j}^{r}g\right>}{r!}
\end{equation}
and because $|\left<h,V_{j}^{r}g\right>|\leq\omega^{k+r}
\|h\|\|f\|_{p}$ we have
\begin{equation}
|F(z)|\leq e^{|z|\omega}\omega^{k}\|h\|\|f\|_{p}.
\end{equation}
For real $t$ we also have $|F(t)|\leq\omega^{k}\|h\|\|f\|_{p}.$
Thus, we proved the inclusion $E_{p}^{\omega}(\mathbb{V})\subset
\mathbb{E}_{p}^{\omega}(\mathbb{V})$.

Now we prove the inverse inclusion by induction. The fact that
$f\in \mathbb{E}_{p}^{\omega}(\mathbb{V})$ means in particular
that for any $1\leq j\leq d$ and any functional $h$ on $ L_{p}(X),
1\leq p\leq \infty,$ the function $F(z)=\left<h,
e^{zV_{j}}f\right>$ is an entire function
 of exponential type $\omega$ which is bounded on the real axis  $\mathbb{R}^{1}$.
 Since $e^{tV_{j}}$ is a group of isometries in $ L_{p}(X)$, an application of
  the Bernstein inequality for functions of one variable gives
$$ \left\|\left<h,
e^{tV_{j}}V_{j}^{m}f\right>\right\|_{C(R^{1})}=\left
\|\left(\frac{d}{dt}\right)^{m}\left<h, e^{tV_{j}}f\right>\right
\|_{ C(R^{1})}\leq\omega^{m}\|h\| \|f\|_{p}, m\in\mathbb{N}.
$$

The last one gives for $t=0$

$$ \left|\left<h, V_{j}^{m}f\right>\right|\leq \omega^{m} \|h\|
 \|f\|_{p}.$$

Choosing $h$ such that $\|h\|=1$ and
\begin{equation}
\left<h, V_{j}^{m}f\right>=\|V_{j}^{m}f\|_{p}
\end{equation}
we obtain the inequality
\begin{equation}
\|V_{j}^{m}f\|_{p}\leq \omega^{m}\|f\|_{p}, m\in \mathbb{N}.
\end{equation}
It was the first step of induction. Now assume that we already
proved that the fact that $f$ belongs to the space
$\mathbb{E}_{p}^{\omega}(\mathbb{V})$ implies the inequality
$$
\|V_{i_{1}}...V_{i_{k}}f\|_{p}\leq \omega^{k}\|f\|_{p}
$$
for any choice of indices $1\leq i_{1},i_{2}...,i_{k}\leq d$. Then
we can apply our first step of induction to the function
$g=V_{i_{1}}...V_{i_{k}}$. It proves the inclusion
$\mathbb{E}_{p}^{\omega}(\mathbb{V})\subset
E_{p}^{\omega}(\mathbb{V})$.

\end{proof}

\begin{thm}The set $E_{p}^{\omega}(\mathbb{V\mathbb{}})$
has the following properties:

1) it is invariant under every $V_{j}$,

2) it is a linear  subspace of $L_{p}(X), 1\leq p\leq \infty$,

3) it is a closed subspace of $L_{p}(X), 1\leq p\leq \infty$.
\end{thm}
\begin{proof}We have to show that if $f\in
E^{\omega}_{p}(\mathbb{V})$ then for any $1\leq
i_{1},i_{2}...,i_{k}, \nu\leq d $ the inequality
\begin{equation}
\|V_{i_{1}}...V_{i_{k}}g\|_{p}\leq \omega^{k}\|g\|_{p},
g=V_{\nu}f,
\end{equation}
holds true. If $f\in E^{\omega}_{p}(\mathbb{V})$, then for any
$V_{\nu}, V_{j}$ and $g=V_{\nu}f$ the inequality
\begin{equation}
\|V_{j}^{k}g\|_{p}\leq\omega^{k+1}\|f\|_{p}=
\omega^{k}\left(\omega\|f\|_{p}\right), k\in \mathbb{N},
\end{equation}
takes place. But then  for any $z\in \mathbb{C}$ we have

$$ \left\|e^{zV_{j}}g\right\|_{p}=
\left\|\sum
^{\infty}_{r=0}\left(z^{r}V_{j}^{r}g\right)/r!\right\|_{p} \leq
\omega\|f\|_{p}\sum ^{\infty}_{r=0}\frac{|z|^{r}\omega^{r}}{r!}=
\omega e^{|z|\omega}\|f\|_{p}.
$$

 As in the proof of the Theorem 2.1  it implies that for any
functional $h$ on $ L_{p}(X), 1\leq p\leq \infty,$ the scalar
function
$$
F(z)=\left<h, e^{zV_{j}}g\right>
$$
is an entire function
 of exponential type $\sigma$ which is bounded on the real axis  $\mathbb{R}^{1}$
by the constant $\|h\| \|g\|_{p}$. An application of the Bernstein
inequality gives the inequality

$$\left\|\left<h, e^{tV_{j}}V_{j}^{k}g\right>\right\|_{C(R^{1})}=
\left \|\left(\frac{d}{dt}\right)^{k}\left<h,
e^{tV_{j}}g\right>\right \|_{ C(R^{1})}
\leq\omega^{k}\|h\|\|g\|_{p}
$$
which leads (see the proof of the Theorem 2.1) to the inequality

\begin{equation}
\|V_{j}^{k}g\|_{p}\leq \omega^{k}\|g\|_{p}, k\in \mathbb{N}.
\end{equation}
It is clear that by repeating these arguments we can prove the
inequality (2.8). The first part of the Theorem  is proved. The
second follows from the fact that the set $
\mathbb{E}^{\omega}_{p}(\mathbb{V}), \omega>0, $ is obviously
linear.

Next, assume that  a sequence $f_{n}\in
E^{\omega}_{p}(\mathbb{V})$ converges in $L_{p}(X)$ to a function
$f$. Because of the Bernstein inequality for any  $1\leq j\leq d$
the sequence $V_{j}f_{n}$ will be fundamental in $L_{p}(X)$. Since
the operator $V_{j}$ is closed the limit of the sequence
$V_{j}f_{n}$ will be the function $V_{j}f$, which implies the
inequality
$$
\|V_{j}f\|_{p}\leq \omega\|f\|_{p}.
$$
By repeating these arguments we can show that if a sequence
$f_{n}\in E^{\omega}_{p}(\mathbb{V})$ converges in $L_{p}(X)$ to a
function $f$ then the Bernstein inequality (2.8) for $f$ holds
true. The Theorem is proved.

\end{proof}

\section{ Paley-Wiener spaces of functions $ PW_{\omega}(X)$ in $L_{2}(X)$}

Let $\textbf{a}^{*}$ be the real dual of $\textbf{a}$ and $W$ be
the Weyl's group. The $\Sigma$ will be the set of all bounded
roots, and $\Sigma^{+}$ will be the set of all positive bounded
roots. The notation $\textbf{a}^{+}$ has the meaning $
\textbf{a}^{+}=\{h\in \textbf{a}|\alpha(h)>0, \alpha\in
\Sigma^{+}\} $
 and is known as positive Weyl's chamber. Let $\rho\in
 \textbf{a}^{*}$ is defined in a way that $2\rho$ is the sum of
 all  positive bounded roots. The Killing form $\left<,\right>$ on $\textbf{g}$
 defines a
 metric on $\textbf{a}$. By duality it defines a scalar product on
 $\textbf{a}^{*}$. The $\textbf{a}^{*}_{+}$ is the set of
 $\lambda\in \textbf{a}^{*}$, whose dual belongs to
 $\textbf{a}^{+}$.

According to Iwasawa decomposition for every $g\in G$ there exists
a unique $A(g)\in \textbf{a}$ such that $g=n \exp A(g) k, k\in K,
n\in N, $
 where $\exp :\textbf{a}\rightarrow A$ is the exponential map of
 the
 Lie algebra $\textbf{a}$ to Lie group $A$. On the direct product
 $X\times \mathcal{B}$ we introduce function with values in $\textbf{a}$
 using the formula $A(x,b)=A(u^{-1}g)$ where $x=gK, g\in G, b=uM, u\in K$.

For every $f\in C_{0}^{\infty}(X)$ the Helgason-Fourier transform
is defined by the formula
$$
\hat{f}(\lambda,b)=\int_{X}f(x)e^{(-i\lambda+\rho)A(x,b))}dx,
$$
where $ \lambda\in \textbf{a}^{*}, b\in \mathcal{B}=K/M, $ and
$dx$ is a $G$-invariant measure on $X$. This integral can also be
expressed as an integral over group $G$. Namely, if $b=uM,u\in K$,
then
\begin{equation}
\hat{f}(\lambda,b)=\int_{G}f(x)e^{(-i\lambda+\rho)A(u^{-1}g))}dg.
\end{equation}
The following inversion formula holds true
\begin{equation}
f(x)=w^{-1}\int_{\textbf{a}^{*}\times
\mathcal{B}}\hat{f}(\lambda,b)e^{(-i\lambda+\rho)(A(x,b))}|c(\lambda)|^{-2}d\lambda
db,
\end{equation}
where $w$ is the order of the Weyl's group and $c(\lambda)$ is the
Harish-Chandra's function, $d\lambda$ is the Euclidean measure on
$\textbf{a}^{*}$ and $db$ is the normalized $K$-invariant measure
on $\mathcal{B}$. This transform can be extended to an isomorphism
between spaces $L_{2}(X,dx)$ and $L_{2}(\textbf{a}^{*}_{+}\times
\mathcal{B}, |c(\lambda)|^{-2}d\lambda db)$ and the Plancherel
formula holds true
\begin{equation}
\|f\|=\left( \int_{\textbf{a}^{*}_{+}\times \mathcal{B}}|\hat{f}
(\lambda,b)|^{2}|c(\lambda)|^{-2}d\lambda db\right)^{1/2}.
\end{equation}

An analog of the Paley-Wiener Theorem hods true that says in
particular that a Helgason-Fourier transform of a compactly
supported distribution is a function which is analytic in
$\lambda$.

It is known, that
\begin{equation}
\widehat{\Delta
f}(\lambda,b)=-(\|\lambda\|^{2}+\|\rho\|^{2})\hat{f}(\lambda,b),f\in
C^{\infty}_{0}(X),
\end{equation}
where $\|\lambda\|^{2}=<\lambda,\lambda>,\|\rho\|^{2}=<\rho,\rho>,
<,>$ is the Killing form on $\textbf{a}^{*}$.

In the case $p=2$ we can introduce  the Paley-Wiener spaces
$PW_{\omega}(X)$ which depend just on the symmetric space $X$.

\begin{defn}  In what follows by  the Paley-Wiener space
$PW_{\omega}(X)$ we understand  the space of all functions $f\in
L_{2}(X)$ whose Helgason-Fourier transform has support in the set
$\left(\textbf{a}^{*}_{+}\right)_{\omega}\times \mathcal{B}$,
where
\begin{equation}
\left(\textbf{a}^{*}_{+}\right)_{\omega}=\left\{\lambda\in
\textbf{a}^{*}_{+}:<\lambda,\lambda>^{1/2}=\|\lambda\|\leq
\omega\right\},\omega\geq 0,
\end{equation}
and $<.,.>$ is the Killing form on $\textbf{a}^{*}$.
\end{defn}

The next theorem is evident.

\begin{thm} The following statements hold true:

\bigskip

1) the set $\bigcup _{ \omega >0}PW_{\omega }(X)$ is dense in
$L_{2}(X)$;

\bigskip

2) the $PW_{\omega }(X)$ is a linear closed subspace in
$L_{2}(X)$.
\end{thm}

We have the following Theorem in which we use notation
$\rho\in\textbf{a}^{*}$ for the half-sum of
 the  positive bounded roots.
\begin{thm}
A function $f$ belongs to $PW_{\omega}(X)$ if and only if
\begin{equation}
\|D^{s}f\|_{2}\leq
\left(\omega^{2}+\|\rho\|^{2}\right)^{s/2}\|f\|_{2}.
\end{equation}
where $D$ is the positive square root from the Laplace-Beltrami
operator $\Delta$, $D=\Delta^{1/2}$.
\end{thm}

\begin{proof}

By using the Plancherel formula and (2.8) we obtain that for every
$\omega$-
  band limited function
  $$
\|\Delta^{\sigma}f\|^{2}=\int_{(\textbf{a}^{*}_{+})_{\omega}}\int_{B}(\|\lambda\|^{2}+\|\rho\|^{2})
^{\sigma} |\widehat{f}(\lambda,b)|^{2}|c(\lambda)|^{2}
dbd\lambda\leq
$$
$$
 (\omega^{2}+\|\rho\|^{2})^{\sigma}
\int_{\textbf{a}^{*}}\int_{B}|\widehat{f}(\lambda,b)|^{2}|c(\lambda)|^{2}
dbd\lambda=(\omega^{2}+\|\rho\|^{2})^{\sigma}\|f\|^{2}.
$$

  Conversely, if $f$ satisfies (3.1), then for any $\varepsilon>0$
  and any $\sigma >0$ we have
  $$
  \int_{\textbf{a}^{*}  \setminus (\textbf{a}^{*}_{+})_{\omega}}\int_{B}
  |\hat{f}(\lambda,b)|^{2}|c(\lambda)|^{-2} dbd\lambda\leq
  $$
  $$
  \int_{\textbf{a}^{*}  \setminus (\textbf{a}^{*}_{+})_{\omega}}\int_{B}
  (\|\lambda\|^{2}+\|\rho\|^{2})^{-2\sigma}
 ( \|\lambda\|^{2}+\|\rho\|^{2})^{2\sigma}|\hat{f}(\lambda,b)|^{2}|c(\lambda)|
 ^{-2} dbd\lambda\leq
  $$
  \begin{equation}
  \left(\frac{\omega^{2}+\|\rho\|^{2}}
  {(\omega+\varepsilon)^{2}+\|\rho\|^{2}}\right)^{2\sigma}
  \|f\|^{2}.
  \end{equation}
It means, that for any $\varepsilon>0$ the function
$\widehat{f}(\lambda,b)$ is zero on $\left\{\textbf{a}^{*}
\setminus (\textbf{a}^{*}_{+})_{\omega}\right\}\times B$. The
statement is proved.

\end{proof}

In a similar way one can prove the following Corollary.

\begin{col}The following statements hold true:

1) the norm of the operator $D=\Delta^{1/2}$ in the space
$PW_{\omega}(X)$ is exactly $\sqrt{\omega^{2}+\|\rho\|^{2}}$;

2) the following limit takes place
$$
\lim_{k\rightarrow
\infty}\|D^{k}f\|_{2}^{1/k}=\sqrt{\omega^{2}+\|\rho\|^{2}},
\>\>\>\>0<\omega<\infty,
$$
if and only if $\>\>\omega$ is the smallest number for which
$(\textbf{a}^{*}_{+})_{\omega}\times B$ contains the support of a
function $\mathcal{F}f, f \in L_{2}(X)$.

\end{col}

 In particular, we
have the following property.

\begin{col}

If a function $f$ belongs to the space $ PW_{\omega}(X)$ then for
any vector fields $V_{i_{1}},...,V_{i_{k}}, V_{j}, 1\leq
i_{1},...,i_{k}, j\leq d$ and any $h\in L_{2}(X)$ the function
\begin{equation}
\int_{X}V_{i_{1}}...V_{i_{k}}f(\exp zV_{j}\cdot
x)\overline{h(x)}dx: \mathbb{C}\rightarrow \mathbb{C}
\end{equation}
is entire function of the exponential type $\leq
\Omega=\sqrt{\omega^{2}+\|\rho\|^{2}}$ which is bounded on the
real line. Conversely, if the function (3.10) is an entire
function of the exponential type
$$
\frac{\Omega}{\sqrt{d}}=\frac{\sqrt{\omega^{2}+\|\rho\|^{2}}}{\sqrt{d}}
$$
for any $V_{i_{1}},...,V_{i_{k}}, V_{j}, 1\leq i_{1},...,i_{k},
j\leq d$ and any $h\in L_{2}(X)$, then $f\in PW_{\omega}(X)$.

\end{col}

The next Lemma describes relations between spaces $PW_{\omega}(X)$ and $E_{\nu}(\mathbb{D})$.  In what follows we assume that $\sqrt{\omega^{2}+\|\rho\|^{2}}>1$.

\begin{lem}\label{inclusion-0} The following statements  hold
\begin{enumerate}

\item there exists a constant $a=a(X)$ such that 
\begin{equation}\label{item1}
PW_{\omega}(X)\subset
E_{a\Omega}(\mathbb{D}),\>\>\Omega=\sqrt{\omega^{2}+\|\rho\|^{2}};
\end{equation}

\item  there exists a constant $b=b(X)$ such that such that 
\begin{equation}\label{item2}
E_{\omega/\sqrt{d}}(\mathbb{D})\subset  PW_{b\omega}(X).
\end{equation}

\end{enumerate}
\end{lem}
\begin{proof} We prove (\ref{item1}). 
Let $A=A(X)$ be a constant  such that for all $f\in H^{1}(X)$ 
\begin{equation}\label{basic norms}
\|D_{j}f\|\leq A\left (\|f\|+\|\Delta^{1/2} f\|\right),\>\>\>1\leq j\leq d.
\end{equation}
Since every $D_{j}$ is a generator of an isometry of $X$ the Laplace-Beltrami operator $\Delta$ commutes with every $D_{j}$. Using (\ref{basic norms})  we obtain the following   inequality for $f\in H^{\infty}(X)$ 
$$
\|D_{j_{1}}D_{j_{2}}...D_{j_{m}}f\|\leq  A\left (\|D_{j_{2}}...D_{j_{m}}f\|+\|D_{j_{2}}...D_{j_{m}}\Delta^{1/2} f\|\right)\leq...
$$
$$\leq
A^{m}\sum_{0\leq l\leq m}C_{m}^{l}\|\Delta^{l/2}f\|\leq (2A)^{m}\sum_{0\leq l\leq m}\|\Delta^{l/2}f\|,
$$
where $C_{m}^{l}$ is the number of combinations from $m$ elements taken $l$ at a time.
Thus, if $f\in PW_{\omega}(X)$ then $\|\Delta^{s}f\|\leq (\omega^{2}+\|\rho\|^{2})^{s}\|f\|$ and
we obtain the inequality
$$
\|D_{j_{1}}...D_{j_{m}}f\|\leq(2A)^{m}\sum_{0\leq l\leq m}(\omega^{2}+\|\rho\|^{2})^{l/2}\|f\|\leq \left(a\sqrt{\omega^{2}+\|\rho\|^{2}}\right)^{m}\|f\|,
$$
where $a=4A$. The inclusion (\ref{item1}) is proved.

Now we prove (\ref{item2}). Let $B=B(X)$ be a constant such that for all $f\in H^{2}(X)$
\begin{equation}
\|\Delta f\|\leq B(\|f\|+\|Lf\|).
\end{equation}
Since  the Laplace-Beltrami operator $\Delta$ commutes with every $D_{j}$  it commutes with $L$ and we have for every $f\in H^{2k}(X)$

\begin{equation}\label{LB-Sum of squares-0}
\|\Delta^{k} f\|\leq B(\|\Delta^{k-1} f\|+\|\Delta^{k-1} Lf\|)\leq...
\leq
(2B)^{k}\sum_{0\leq l\leq k}\|L^{l}f\|.
\end{equation}
Using definition of the operator $L$  one can easily verify that for any
natural $l$  the function
$L^{l}\left(f\right)$ is a sum of $\>\>d^{l}$ terms of the following form:
\begin{equation}
D_{j_{1}}^{2}...D_{j_{l}}^{2}(f),\ 1\leq j_{1},...,j_{l}\leq d.
\end{equation}
Thus, for  $f\in E_{\omega/\sqrt{d}}(\mathbb{D})$  one has  
\begin{equation}
\|D_{i_{1}}...D_{i_{k}}f\|\leq\left(\frac{ \omega}{\sqrt{d}}\right)^{k}\|f\|,
\end{equation}
and then
\begin{equation}
\|\Delta^{k} f\|\leq 
(2B)^{k}\sum_{0\leq l\leq k}\|L^{l}f\|\leq 
$$
$$
(2B)^{k}\sum_{0\leq l\leq k}\sum_{1\leq j_{1},...,j_{l}\leq d}\|D_{j_{1}}^{2}...D_{j_{l}}^{2}(f)\| \leq \left(b(\omega^{2}+\|\rho\|^{2})\right)^{k}\|f\|,
\end{equation}
where $b=4B$.  Lemma  is proved.

\end{proof}

\section{Embedding Theorems }

 Denote by $T_{x}(X)$ the
tangent space of $X$ at a point $x\in X$ and let $ exp_ {x} $ :
$T_{x}(X)\rightarrow X$ be the exponential geodesic map i.  e.
$exp_{x}(u)=\gamma (1), u\in T_{x}(X)$ where $\gamma (t)$ is the
geodesic starting at $x$ with the initial vector $u$ : $\gamma
(0)=x , \frac{d\gamma (0)}{dt}=u.$ We will always assume that all
our local coordinates are defined by $exp$.

We consider a uniformly bounded partition of unity
$\{\varphi_{\nu}\}$ subordinate to a cover of  $X$ of finite
multiplicity
$$
X=\bigcup_{\nu} B(x_{\nu}, r),
$$
where $B(x_{\nu}, r)$ is a metric ball at  $x_{\nu}\in X$ of
radius $r$ and introduce the  Sobolev space $W^{k}_{p}(X), k\in
\mathbb{N}, 1\leq p<\infty,$ as the completion of
$C_{0}^{\infty}(X)$ with respect to the norm
\begin{equation}
\|f\|_{W^{k}_{p}(X)}=\left(\sum_{\nu}\|\varphi_{\nu}f\|^{p}
_{W^{k}_{p}(B(y_{\nu}, r))}\right) ^{1/p}.
\end{equation}

 The regularity theorem for $\Delta$ means in
particular, that the norm of the Sobolev space $W_{p}^{2k}(X),
k\in \mathbb{N}, 1\leq p< \infty,$ is equivalent to the graph norm
$\|f\|_{p}+\|\Delta^{k}f\|_{p}$.

Since  vector fields $V_{1},...,V_{d}$,  generate the tangent
space at every point of $X$ the norm of the space $W_{p}^{2k}(X),
k\in \mathbb{N}, 1\leq p< \infty,$ is equivalent to the norm
\begin{equation}
\|f\|_{p}+\sum_{j=1}^{k} \sum_{1\leq i_{1},...,i_{j}\leq
d}\|V_{i_{1}}...V_{i_{j}}f\|_{p}, 1\leq p<\infty.
\end{equation}
Using the closed graph Theorem and the fact that every $V_{i}$ is
a closed operator in $L_{p}(X), 1\leq p<\infty,$ it is easy to
show that the norm (4.2) is equivalent to the norm
\begin{equation}
\|f\|_{p}+\sum_{1\leq i_{1},..., i_{k}\leq
d}\|V_{i_{1}}...V_{i_{k}}f\|_{p}, 1\leq p<\infty.
\end{equation}

 Let $\pi :G\rightarrow X=G/K$ be the natural projection and $o\in
 X$ is the
image of identity in $G$. We consider a ball $B(o,r/4)$ in the
invariant metric on $X$. Now we choose such elements $g_{\nu}\in
G$ that the family of balls $B(x_{\nu},r/4), x_{\nu}=g_{\nu}\cdot
o,$ has the following maximal property: there is no ball in $X$ of
radius $r/4$ which would have empty intersection with every ball
from this family. Then the balls of double radius $B(x_{\nu},r/2)$
would form a cover of $X$. Of course, the balls $B(x_{\nu},r)$
will also form a cover of $X$. Let us estimate the multiplicity of
this cover.

Note, that the Riemannian volume $B(\rho)$ of a ball of radius
$\rho$ in $X$ is independent of its  center and   is given by the
formula
$$
B(\rho)=\int_{0}^{\rho}S(t)dt,
$$
where the surface area $S(t)$ of a sphere of radius $t$.

 Every ball from the family $\{B(x_{\nu}, r)\}$, that has
non-empty intersection with a particular ball $B(x_{j}, r)$ is
contained in the ball $B(x_{j}, 3r)$. Since any two balls from the
family $\{B(x_{\nu}, r/4)\}$ are disjoint, it gives the following
estimate for the index of multiplicity $N$ of the cover
$\{B(x_{\nu}, r)\}$:
$$
N\leq\frac{B(3r)}{B(r/4)}\leq\frac{\int_{0}^{3r}S(t)
dt}{\int_{0}^{r/4} S(t) dt}.
$$

By using some elementary inequalities for the function $sh$ one
can obtain the following rough estimate
$$
N\leq 12^{d}e^{\sqrt{d-1}}.
$$

 So, we proved the following Lemma.
\begin{lem}
For any   $r>0$ there
 exists a set of points $\{x_{\nu}\}$ from $X$ such that

1) balls $B(x_{\nu}, r/4)$ are disjoint,

2) balls $B(x_{\nu}, r/2)$ form a cover of $X$,

3) multiplicity of the cover by balls $B(x_{\nu}, r)$ is not
greater $N_{d}=(12)^{d}e^{\sqrt{d-1}}.$

\end{lem}

\begin{defn}
We will use notation $Z(x_{\nu}, r, N_{d})$ for a set of points
$\{x_{\nu}\}\in X$ which satisfies the properties 1)- 3) from the
last Lemma and we will call such set a  $(r,N_{d})$-lattice in
$X$.
\end{defn}

\begin{defn}
We will use notation $Z_{G}(g_{\nu}, r, N_{d})$ for a set of
elements $\{g_{\nu}\}$ of the group $G$  such that the points
$\{x_{\nu}=g_{\nu}\cdot o\}$ form a $(r,N_{d})$-lattice in $X$
(here $\{o\}\in X$ is the origin of $X$). Such set $Z_{G}(g_{\nu},
r, N_{d})$ will be  called a $(r,N_{d})$-lattice in $G$.
\end{defn}

\begin{thm}
For any $(r,N_{d})$-lattice $Z_{G}(g_{\nu}, r, N_{d})\subset
 G$, any $m>d/p$ there exists constants $C(X, N_{d})$ and $C(X, N_{d}, m)$
such that for
 any $\omega>0$ and any $1\leq p<q\leq \infty$ the following
 inequalities hold true
\begin{equation}
 \|f\|_{q}\leq C(X)r^{d/p} \sup_{g\in
G}\left(\sum_{i}\left(|f(g_{i}g\cdot
o)|\right)^{p}\right)^{1/p}\leq
$$
$$
C(X,m)r^{d/q-d/p}\left(1+(r\omega)^{m}\right)\|f\|_{p},
\end{equation}
for all $f\in E^{\omega}_{p}(X)$.

 In particular the following
embeddings hold true
\begin{equation}
E^{\omega}_{p}(\mathbb{V})\subset L_{q}(X),
\mathbb{V}=\left\{V_{1},...,V_{d}\right\},
\end{equation}
for any $1\leq p\leq q\leq\infty$.

\end{thm}

\begin{proof}

In what follows we fix a  $r>0$ and consider a cover of $X$ of
finite multiplicity $N_{d}$ by balls $\left\{B(g_{i}\cdot
o,r)\right\},$ which was constructed in Lemma 4.1. First we are
going to use the following inequality

\begin{equation}
|\psi(y)|\leq C_{1}(d,m)\sum_{0\leq j\leq
m}r^{j-d/p}\|\psi\|_{W_{p}^{j}(B(g_{i}\cdot o, r))}, m>d/p,
\end{equation}
  where
$y\in B(g_{i}\cdot o, r/2), \psi \in C^{\infty}\left(B(g_{i}\cdot
o, r)\right)$. From the inequality (4.6) we have for $1\leq
p<\infty$
\begin{equation}
\left(r^{d/p}|f(g_{i}\cdot o)|\right)^{p}\leq
C_{2}(d,m)\sum_{0\leq j\leq m}r^{jp}\|f\|_{W_{p}^{j}(B(g_{i}\cdot
o, r))}^{p}, m>d/p,
\end{equation}
and
$$
\sum_{i}\left(r^{d/p}|f(g_{i}\cdot o)|\right)^{p}\leq
C_{3}(d,m)\sum_{i}\sum_{0\leq j\leq
m}r^{jp}\|f\|_{W_{p}^{j}(B(g_{i}\cdot o, r))}^{p}, m>d/p.
$$

 We obtain
that for any given $m>d/p$ there exists a constant
 $C(X,N_{d},m)>0$, such that for any $(r,N_{d})$-lattice
  the following inequality holds true for $1\leq
p<\infty$

$$
\left(\sum_{i}\left(r^{d/p}|f(g_{i}\cdot
o)|\right)^{p}\right)^{1/p}\leq
C(X,N_{d},m)\left(\|f\|_{p}+r^{j}\sum_{j=1}^{m}\|f\|_{W^{j}_{p}(X)}\right).
$$
Since the vector fields $V_{1},...,V_{d},$ form a basis of the
tangent space at every point of $X$ the Sobolev norm
$\|f\|_{W^{k}_{p}(X)}$ for every $k\in \mathbb{N}$ is equivalent
to the norm
$$
\|f\|_{p}+\sum_{j=1}^{m} \sum_{0\leq k_{1},...,k_{j}\leq
d}\|V_{k_{1}}...V_{k_{j}}f\|_{p}, 1\leq p<\infty.
$$
We obtain

\begin{equation}
\left(\sum_{i}\left(r^{d/p}|f(g_{i}\cdot
o)|\right)^{p}\right)^{1/p}\leq
$$
$$
C(X,N_{d},m)\left(\|f\|_{p}+\sum_{j=1}^{m} \sum_{0\leq
k_{1},...,k_{j}\leq d}r^{j}\|V_{k_{1}}...V_{k_{j}}f\|_{p}\right),
m>d/p.
\end{equation}
Because every $V_{k}, k=1,...,d,$ is a generator of a
one-parameter isometric group of bounded operators in $L_{p}(X)$,
the following interpolation inequality holds true
\begin{equation}
r^{l}\|V^{l}_{k}f\|_{p}\leq
a^{m-l}r^{m}\|V_{k}^{m}f\|_{p}+c_{m}a^{-l}\|f\|_{p}, 1\leq
p<\infty,
\end{equation}
for any $1\leq l<m, a, r >0.$ The last two inequalities imply the
following estimate
\begin{equation}
\left(\sum_{i}(r^{d/p}|f(g_{i}\cdot o)|)^{p}\right)^{1/p}\leq
$$
$$
C(X, N_{d}, m)\left(\|f\|_{p}+r^{m} \sum_{0\leq
k_{1},...,k_{m}\leq d}\|V_{k_{1}}...V_{k_{m}}f\|_{p}\right),
m>d/p.
\end{equation}
 For $f\in E_{p}^{\omega}(\mathbb{V})$ it gives for $1\leq
p<\infty$
 \begin{equation}
\left(\sum_{i}\left(r^{d/p}|f(g_{i}\cdot
o)|\right)^{p}\right)^{1/p}\leq
C(X,N_{d},m)\left(1+(r\omega)^{m}\right)\|f\|_{p}, m>d/p.
\end{equation}
 Applying this inequality to a translated function
$f(h\cdot x), h\in G,$ and using invariance of the measure $dx$ we
obtain for $f\in E_{p}^{\omega}(\mathbb{V})$

$$
\sup_{h\in G}\left(\sum_{i}\left(r^{d/p}\left|f(hg_{i}\cdot
o)\right|\right)^{p}\right) ^{1/p}\leq
$$
\begin{equation}
 C(X,m)\left(1+(r\omega)^{m}\right)\sup_{h\in
G}\left\|f(h\cdot x)\right\|_{p}=
C(X,m)\left(1+(r\omega)^{m}\right)\|f\|_{p}.
\end{equation}

We introduce the following neighborhood of the identity in the
group $G$
$$
Q_{r}=\left\{g\in G : g\cdot o\in B(o,r)\right\}.
$$
According to the known formula

$$
\int_{X}f(x)dx=\int_{G}f(g\cdot o)dg, f\in C_{0}(X),
$$
we have for the characteristic function $\chi_{B}$ of the ball
$B(o,r)$
$$
r^{d}\approx\int_{B(o,r)}dx=
\int_{X}\chi_{B}(x)dx=\int_{G}\chi_{B}(g\cdot o)dg=\int_{Q_{r}}dg.
$$
Thus, since every ball in our cover is a translation of the ball
$B(o, r)$ by using $G$-invariance of the measure $dx$ we obtain
for any $f\in L_{q}(X), 1\leq q\leq \infty,$

$$\int_{X}|f(x)|^{q}dx\leq
N_{d}\sum_{i}\int_{B(g_{i}\cdot o,r)}|f(x)|^{q}dx\leq
N_{d}\sum_{i}\int_{B(o,r)}|f(g_{i}\cdot y)|^{q}dy=
$$
$$
C(X,N_{d})\int_{Q_{r}}\sum_{i}|f(g_{i}h\cdot o)|^{q}dh\leq
C(X,N_{d})r^{d}\sup_{g\in G}\sum_{i}|f(g_{i}g\cdot o)|^{q},
$$
where $d=dim X$. After all  we have
$$\|f\|_{q}\leq C(X,N_{d}) r^{d/q}\sup_{g\in
G}\left(\sum_{i}\left(|f(g_{i}g\cdot
o)|\right)^{q}\right)^{1/q}f\in L_{q}(X), 1\leq q\leq \infty.
$$

Next, using the inequality
$$
\left(\sum a_{i}^{q}\right)^{1/q}\leq \left(\sum
a_{i}^{p}\right)^{1/p},
$$
which holds true for any $ a_{i}\geq 0, 1\leq p\leq q\leq \infty,$
we obtain the following inequality

\begin{equation}
\|f\|_{q}\leq C(X,N_{d})r^{d/q} \sup_{g\in
G}\left(\sum_{i}\left(|f(g_{i}g\cdot
o)|\right)^{q}\right)^{1/q}\leq
$$
$$C(X,N_{d})r^{d/q} \sup_{g\in
G}\left(\sum_{i}\left(|f(g_{i}g\cdot o)|\right)^{p}\right)^{1/p}=
$$
$$C(X,N_{d})r^{d/q-d/p} \sup_{g\in
G}\left(\sum_{i}\left(r^{d/p}|f(g_{i}g\cdot
o)|\right)^{p}\right)^{1/p} .
\end{equation}

From the
 inequalities (4.12) and (4.13) and the observation, that
 for the element $ g=g_{i}^{-1}hg_{i} $ the expression
$$
\sum_{i}\left(r^{d/p}\left|f(g_{i}g\cdot o)\right|\right)^{p}
$$
becomes the expression
$$\sum_{i}\left(r^{d/p}\left|f(hg_{i}\cdot
o)\right|\right)^{p},
$$
we obtain the Theorem 4.2.

\end{proof}

As a consequence we have the following Corollary.

\begin{col}
For any $1\leq  p\leq q\leq \infty$ the following  embeddings hold
true
$$
E^{\omega}_{p}(\mathbb{V})\subset E^{\omega}_{q}(\mathbb{V}).
$$
In particular, the  spaces $E_{q}^{\omega}(\mathbb{V})$ are not
trivial at least if $2\leq q\leq \infty$ and $\omega\geq\|\rho\|$.
\end{col}

Here the notation $\rho\in \textbf{a}^{*}$  means the half-sum of
 all  positive restricted roots.
\begin{proof}
Since $ E^{\omega}_{p}(\mathbb{V})$ is invariant under every
operator $V_{i}, 1\leq i\leq d,$ it is enough to show that if
$f\in E^{\omega}_{p}(\mathbb{V})$, then for any $1\leq j\leq d,
k\in \mathbb{N}$
$$
\|V_{j}^{k}f\|_{q}\leq \omega^{k}\|f\|_{q}.
$$
We are using the same arguments as in the proof of the Theorem
2.2. Namely, since $f\in E^{\omega}_{p}(\mathbb{V})$ we have for
any $z\in \mathbb{C}$

$$ \left\|e^{zV_{j}}f\right\|_{q}=
\left\|\sum
^{\infty}_{r=0}\left(z^{r}V_{j}^{r}f\right)/r!\right\|_{q} \leq
e^{|z|\omega}\|f\|_{p}.
$$

 As it was shown in the proof of the Theorem 2.1  it implies that for any
functional $h$ on $ L_{q}(X), 1\leq p\leq \infty,$ the scalar
function
$$
F(z)=\left<h, e^{zV_{j}}f\right>,
$$
is an entire function
 of exponential type $\omega$ which is bounded on the real axis  $\mathbb{R}^{1}$
by the constant $\|h\| \|f\|_{p}$. The classical Bernstein
inequality gives
$$ \sup_{t}\left|\left<h,
e^{tV_{j}}V_{j}^{k}f\right>\right|=\sup_{t}\left
|\left(\frac{d}{dt}\right)^{k}\left<h, e^{tV_{j}}f\right>\right
|\leq\omega^{k}\|h\| \|f\|_{q}, m\in\mathbb{N}.
$$

When $t=0$ we obtain

$$ \left|\left<h, V_{j}^{k}f\right>\right|\leq \omega^{k} \|h\|
 \|f\|_{q}.$$

Choosing $h$ such that $\|h\|=1$ and
\begin{equation}
\left<h, V_{j}^{k}f\right>=\|V_{j}^{k}f\|_{q}
\end{equation}
we get the inequality
\begin{equation}
\|V_{j}^{k}f\|_{q}\leq \omega^{k}\|f\|_{q}, k\in \mathbb{N}.
\end{equation}
The Corollary 4.1 is proved.

\end{proof}
Inequalities (4.4) and (4.5) are known as Nikolskii inequalities.
Now we are ready to prove a generalization of another inequality
which is also attributed to Nikolskii.

\begin{thm}
 There exists a constant $C(X)$
 such that for any $1\leq p\leq q\leq \infty$
the following  inequality holds true for all $f\in
E_{p}^{\omega}(\mathbb{V})$

\begin{equation}
\|f||_{q}\leq C(X)\omega^{\frac{d}{p}-\frac{d}{q}}\|f\|_{p}, d=dim
X.
\end{equation}

\end{thm}
\begin{proof} The Theorem 4.2 imply that for any  $(r,N_{d})$-lattice, any
  $m>d/p$ there exists a constant $C(X,N_{d},m)>0$  such that
 for any $\omega>0$ and any $1\leq p\leq q\leq \infty$ the following inequality holds
 true
 $$
 \|f\|_{q}\leq C(X,N_{d},m)r^{d/q-d/p}(1+(r\omega)^{m})\|f\|_{p},
 $$
 for all $f\in E^{\omega}_{p}(X)$.
We make the substitution $t=r\omega$ into this inequality  to
obtain

\begin{equation}
\|f\|_{q}\leq
C(X,N_{d},m)\omega^{d/p-d/q}\left(t^{d/q-d/p}(1+t^{m})\right)\|f\|_{p}=
$$
 $$
C(X,N_{d},m)\eta_{p,q}(t)\omega^{d/p-d/q}\|f\|_{p}, m>d/p,
\end{equation}
where
$$
\eta_{p,q}(t)=t^{d/q-d/p}(1+t^{m}), t\in(0,\infty).
$$

Since $m$ can be any number greater than $d/p$ and $p\geq 1,$ we
fix the number $m=2d$. At the point
\begin{equation}
t_{d,p,q}=\frac{\alpha}{2d-\alpha}\in (0,1),
\end{equation}
where $0< \alpha=d/p-d/q< 1,$ the function $\eta_{p,q}$ has its
minimum, which is
$$
\eta_{p,q}(t_{d,p,q})=\frac{1}{(1-\beta)^{1-\beta}\beta^{\beta}}\leq
2,
$$
where $\beta=\alpha/2d$.  Thus, if we would substitute this
$t_{d,p,q}$ into (4.17) we would have for any   $\omega>0$ and any
$1\leq p\leq q\leq \infty$ the desired inequality
$$
\|f\|_{q}\leq
2C(X,N_{d})\omega^{\frac{d}{p}-\frac{d}{q}}\|f\|_{p}, d=dim X,
1\leq p\leq q\leq \infty.
$$

For a given $d\in \mathbb{N}, \omega>0, 1\leq p\leq q\leq \infty,$
we can find corresponding $t_{d,p,q}$ using the formula (4.18) and
then can find the  corresponding $r>0$ by using the formula
\begin{equation}
r=r_{d,p,q, \omega}=\frac{t_{d,p,q}}{\omega}.
\end{equation}
According to the Lemma 4.1 for such $r$ from (4.19) one can find a
cover of the same multiplicity $N_{d}$.  For this cover  we will
have the inequality (4.16). The Theorem is proved.

\end{proof}

The next goal is to show Plancherel-Polya-type inequalities (1.8).

For a fixed $(r, N_{d})$-lattice $Z(x_{\nu}, r, N_{d})$ (see
Definition 4) we consider the following set
$\Phi=\left\{\Phi_{\nu}\right\}$ of distributions $\Phi_{\nu}$.

 Let $K_{\nu}\subset B(x_{\nu},r/2)$ be a compact
subset and $\mu_{\nu}$ be a non-negative measure on $K_{\nu}$. We
will always assume that the total measure of $K_{\nu}$ is finite,
i.e.
$$
0<|K_{\nu}|=\int_{K_{\nu}}d\mu_{\nu}<\infty.
$$

We consider the following distribution on
$C_{0}^{\infty}(B(x_{\nu},r)),$
\begin{equation}
\Phi_{\nu}(\varphi)=\int_{K_{\nu}}\varphi d\mu_{\nu},
\end{equation}
where $\varphi \in C_{0}^{\infty}(B(x_{\nu},r)).$ As a compactly
supported distribution of order zero it has a unique continuous
extension to the space $C^{\infty}(B(x_{\nu}, r))$.

\bigskip

We say that a family $\Phi=\{\Phi_{\nu}\}$ is uniformly bounded,
if there exists a positive constant $C_{\Phi}$ such that
\begin{equation}
|K_{\nu}|\leq C_{\Phi}
 \end{equation}
 for all $\nu$.

We will also say that a family $\Phi=\{\Phi_{\nu}\}$ is separated
from zero if there exists a constant $c_{\Phi}>0$ such that
\begin{equation}
|K_{\nu}|\geq c_{\Phi}
 \end{equation}
  for all $\nu$.

\bigskip

The next goal is to obtain the Plancherel-Polya inequalities for
functions from $E_{p}^{\omega}(\mathbb{V})$.

\begin{thm}For any given $m>d/p$ and $C_{\Phi}>0$ there exist
$C_{1}=C_{1}(N_{d}, C_{\phi}, m)>0,$ $ C_{2}=C_{2}(N_{d},
 C_{\Phi}, m)>0, r(N_{d},
C_{\Phi}, m)>0$ such that for any $(r,N_{d})$-lattice $Z(x_{\nu},
r, N_{d})$ with $ 0<r<r(N_{d}, C_{\Phi}, m),$ for any family
$\{\Phi_{\nu}\}$ of distributions of type (4.20) with supports in
$B(x_{\nu}, r/2)$ which satisfy (4.21) and (4.22) with given
$C_{\Phi}$ the following inequalities hold true

\begin{equation}
\left(\sum_{\nu}|\Phi_{\nu}(f)|^{p}\right)^{1/p}\leq
C_{1}r^{-d/p}\left(\|f\|_{p}+\sum_{1\leq k_{1},...,k_{m}\leq
d}\|V_{k_{1}}...V_{k_{m}}f\|_{p}\right),
 \end{equation}
$$
\|f\|_{p}\leq
$$
\begin{equation}
 C_{2}\left\{r^{d/p}\left(
\sum_{\nu}|K_{\nu}|^{-1}|\Phi_{\nu}(f)|^{p}\right)^{1/p}+
c_{\Phi}^{-1}r^{m} \sum_{1\leq k_{1},...,k_{m}\leq
d}\|V_{k_{1}}...V_{k_{m}}f\|_{p} \right\}.
\end{equation}
\end{thm}

The similar Theorem was proved in \cite{Pes4} in the case $p=2$.
In what follows we just sketch the proof.

\begin{proof}

The inequality (4.23) follows from the definitions of
 the distributions $\Phi_{\nu}$ and the inequality
\begin{equation}
|\psi(y)|\leq C_{0}(d,k)\sum_{0\leq j\leq
k}r^{j-d/p}\|\psi\|_{W_{p}^{j}(B(x_{\nu},r))}, k>d/p,
\end{equation}
where $y\in B(x_{\nu},r/2), \psi \in C^{\infty}(B(x_{\nu},r))$.

To prove (4.24) we show that for any $k>d/p$ there exists a
constant $C=C(d, k)>0, $
 such that for  any ball $B(x_{\nu}, r), x_{\nu}\in
 M,$ any distribution $\Phi_{\nu}$ of type (4.20)
 the following inequality holds true

\begin{equation}
\left\|f-|K_{\nu}|^{-1}\Phi_{\nu}(f)\right\|_{L_{p}(B(x_{\nu},r/2)}\leq
C(d,k)\sum_{1\leq |\alpha|\leq
k}r^{|\alpha|}\|\partial^{|\alpha|}f\|_{L_{p}(B(x_{\nu},r))},
\end{equation}
where  $f\in W_{p}^{k}(M), k>d/p, 1\leq p\leq \infty,$ and
$\partial^{j}f$ is a partial derivative of order $j$.

 To prove the inequality (4.26) we make use of the Taylor series. For any $f\in
C^{\infty}(B(x_{\nu},r/2))$, every $x,y\in B(x_{\nu},r/2)$ we have
the following
$$
f(x)=f(y)+\sum_{1\leq|\alpha|\leq k-1} \frac{1}{\alpha
!}\partial^{|\alpha|}f(y)(x-y) ^{\alpha}+
$$
$$
\sum_{|\alpha|=k}\frac{1}{\alpha !}\int_{0}^{\eta}t^{k-1}\partial
^{|\alpha|}f(y+t\vartheta)\vartheta^{\alpha}dt,
$$
where $x=(x_{1},...,x_{d}), y=(y_{1},...,y_{d}), \alpha=(
\alpha_{1},...,\alpha_{d}),
(x-y)^{\alpha}=(x_{1}-y_{1})^{\alpha_{1}}...
(x_{d}-y_{d})^{\alpha_{d}}, \eta=\|x-y\|, \vartheta=(x-y)/\eta.$

We  integrate each term over compact $K_{\nu}\subset B(x_{\nu},r)$
against $d\mu_{\nu}(y)$, where $d\mu_{\nu}$ is the measure on
$K_{\nu}$. After all we obtain
$$
\left\|f-|K_{\nu}|^{-1}\Phi_{\nu}(f)\right\|_{L_{p}(B(x_{\nu},r/2))}\leq
$$
$$
C(k,d)|K_{\nu}|^{-1}\sum_{1\leq|\alpha|\leq
k-1}\left(\int_{B(x_{\nu},r/2)}\left(
\int_{K_{\nu}}\left|\partial^{|\alpha|}f(y)(x-y)
^{\alpha}\right|d\mu_{\nu}(y)\right)^{p}dx\right)^{1/p}+
$$
\begin{equation}
C(k,d)|K_{\nu}|^{-1}
\sum_{|\alpha|=k}\left(\int_{B(x_{\nu},r/2)}\left(\int_{K_{\nu}}
\left|\int_{0}^{\eta}t^{k-1}\partial
^{|\alpha|}f(y+t\vartheta)\vartheta^{\alpha}dt\right|d\mu_{\nu}(y)\right)^{p}dx\right)^{1/p}.
\end{equation}

By using the Minkowski inequality and the estimate (4.25) we
obtain (4.26). Summation over all $\nu$ gives the inequality
(4.24).
\end{proof}

The Theorems 4.2, 4.4 and the Bernstein inequality (1.5) imply the
following Plancherel-Polya inequalities.

\begin{thm}

For any given $\omega>0, C_{\Phi}>0, c_{\Phi}>0, m=0,1,2,...,$
there exist positive constants $C,  c_{1}, c_{2},$ such that for
every $\rho$-lattice $Z(x_{\nu},r , N_{d})$ with $0<r<
(C\omega)^{-1}$, every family of distributions $\{\Phi_{\nu}\}$ of
the form (4.20) with properties (4.21), (4.22) and every $f\in
E^{\omega}_{p}(X)$ the following inequalities hold true

\begin{equation}
c_{1}\left(\sum_{\nu}\left|\Phi_{\nu}(f)\right|^{p}\right)^{1/p}\leq
r^{-d/p}\|f\|_{p} \leq c_{2}\left(\sum_{\nu}
|\Phi_{\nu}(f)|^{p}\right)^{1/p}.
\end{equation}

\end{thm}

\bigskip

 In the case of Euclidean
space when $\Phi_{\nu}=\delta_{x_{\nu}}$ and $\{x_{\nu}\}$ is the
regular lattice the above inequality represents the classical
Plancherel-Polya inequalities.

The notation $l_{p}^{\omega}$ will be used for a linear subspace
of all sequences $\{v_{\nu}\}$ in $l_{p}$ for which there exists a
function $f$ in $E^{\omega}_{p}(X)$ such that
$$
\Phi_{\nu}(f)=v_{\nu}, \nu\in \mathbb{N}.
$$
In general $l_{p}^{\omega}\neq l_{p}$.

\begin{defn}
A linear reconstruction method $R$ is a linear operator
$$
R:l_{p}^{\omega}\rightarrow E^{\omega}_{p}(X)
$$
such that
$$
R: \left\{\Phi_{\nu}(f)\right\}\rightarrow f.
$$

 The reconstruction method is said to be stable, if it is
continuous in topologies induced respectively by $l_{p}$ and
$L_{p}(X)$.
\end{defn}

The following result is a  consequence of the Plancherel-Polya
inequalities and the fact that $E^{\omega}_{p}(X), 1\leq
p\leq\infty,$ is a linear space.

\begin{thm}

For any given $\omega>0, C_{\Phi}>0, c_{\Phi}>0, m=0,1,2,...,$
there exist positive constants $C,  c_{1}, c_{2},$ such that for
every $\rho$-lattice $Z(x_{\nu},r , N_{d})$ with $0<r<
(C\omega)^{-1}$, every family of distributions $\{\Phi_{\nu}\}$ of
the form (4.20) with properties (4.21), (4.22) and every $f\in
E^{\omega}_{p}(X)$ the following statements hold true

\bigskip
1)  every function $f$ from $E^{\omega}_{p}(X), 1\leq
p\leq\infty,$ is uniquely determined by the set of samples
$\left\{\Phi_{\nu}(f)\right\}$;

\bigskip

2)  reconstruction method $R$ from a set of samples
$\left\{\Phi_{\nu}(f)\right\}$
$$
R:\left\{\Phi_{\nu}(f)\right\}\rightarrow f
$$
is stable.
\end{thm}

\makeatletter \renewcommand{\@biblabel}[1]{\hfill#1.}\makeatother

\end{document}